\documentclass{amsart}
\usepackage{amssymb, amsfonts, amsmath, amsthm}

\DeclareMathSymbol{\twoheadrightarrow}  {\mathrel}{AMSa}{"10}

\DeclareMathSymbol{\twoheadrightarrow} {\mathrel}{AMSa}{"10}

                                            \def\const{\mathrm{const}}

                 \def\rg{\mathrm{rg}}

\def\C{{\mathbb C}}

\def\M{{\mathcal M}}

              \def\Mult{\mathrm{Mult}}

        \def\K_a{\bar{K}}

\def\M{\mathrm{M}}

\def\dim{\mathrm{dim}}

\def\K{{\mathcal{K}}}

\newtheorem{thm}{Theorem}[section]

\newtheorem{cor}[thm]{Corollary}

\theoremstyle{definition}

\newtheorem{ex}[thm]{Example}
\newtheorem{rem}[thm]{Remark}

\title[Polynomials  and  tangent maps]{Polynomials in one variable  and ranks of certain tangent maps}

\author{Yuri G. Zarhin}

\address{Department of Mathematics, Pennsylvania
State University, University Park, PA 16802, USA}

\address{Institute for Mathematical Problems in Biology,
Russian Academy of Sciences, Pushchino, Moscow Region, Russia}

\email{zarhin\char`\@math.psu.edu}

\begin{document}
\begin{abstract}
We study a map that sends a monic degree $n$ complex polynomial
$f(x)$ without multiple roots to the collection of $n$ values of its
derivative at the roots of $f(x)$. We give an answer to a question
posed by Yu.S. Ilyashenko.

\end{abstract}

\maketitle

\section{Definitions, Notation, Statements}
We write $\C$ for the field of complex numbers. Our aim is to
compute the rank of the following map.

Let us consider the $n$-dimensional complex manifold
$P_n\subset\C^n$ of all monic complex polynomials of degree  $n\ge
2$
$$f(x)=x^n+\sum_{i=0}^{n-1} a_i x^i$$ with coefficients
 $a=(a_0, \dots , a_{n-1})$  and without multiple roots. We denote
 the roots of
 $f(x)$ by
$$\alpha=\{\alpha_1, \dots , \alpha_n\},$$
somehow ordering them; locally (with respect to $a$), one may choose
each  $\alpha_i$  (using Implicit Function Theorem) as a smooth
(univalued) function in  $a$. Further,  we will try to differentiate
these functions with respect to coordinates, without computation of
the roots. And here is our map
$$M:a=(a_0, \dots , a_{n-1}) \mapsto f^{\prime}(\alpha)=(f^{\prime}(\alpha_1), \dots f^{\prime}(\alpha_n)) \in \C^n .$$
Abusing notation, we may assume that $M$  is defined locally on
$P_n$ and write $M(f)$ instead of $M(a_0, \dots , a_{n-1})$. Let
 $$dM: \C^n \to \C^n$$ be the corresponding tangent map
  (at the point $f(x)$). It is convenient to identify the tangent
  space
   $\C^n$  with the space of all polynomials
  $p(x)$ of degree $\le n-1$. Namely, to a polynomial $p(x)=\sum_{i=0}^{n-1} c_i x^i$
 one assigns the tangent vector $(c_0, \dots, c_{n-1}) \in \C^n$.
 So, to the derivative $f^{\prime}(x)$ corresponds the tangent
 vector
$(a_1, \dots , (n-1)a_{n-1},n)\in \C^n$.
 We denote by $W=W(f)$ the complex vector space of all polynomials $p(x)$ of degree $\le
 n-2$ such that the polynomial
 $$R(x):=f^{\prime\prime}(x)p(x)-f^{\prime}(x)p^{\prime}(x)$$ is divisible by
 $f(x)$.

\begin{thm}
\label{main1} The rank of the tangent map $dM: \C^n \to \C^n$ is
 $n-1$ at all points of $P_n$. The kernel of $dM$ always
contains $f^{\prime}(x)$ and coincides with the one-dimensional
subspace $\C\cdot f^{\prime}(x)$.
\end{thm}

\begin{rem}
\label{period} The non-triviality of the kernel of
 $dM$  is related to the fact that  for
 each (small) complex number
 $\epsilon$ the map $M$ sends
 $f(x)$ and $f(x+\epsilon)$ to the same vector in $\C^n$.
\end{rem}

Denote by $H_{n}$  the space of all monic complex polynomials
 $g(x)$ of degree
$n\ge 2 $ such that the map
$$G: \C \to \C, z \mapsto g(z)$$
has exactly $n$ fixed points. Clearly,  $g \in H_n$ if and only if
$g(x)-x$ has no multiple roots, i.e.,
$$g(x)-x\in P_n.$$
It is also clear that the roots  $\beta_1, \dots , \beta_n$ of
 $f(x)=g(x)-x$ are exactly the fixed points of $G$ and the corresponding  {\sl  multiplier}
 $g^{\prime}(\beta_i)$ of $\beta_i$ coincides with
 $f(\beta_i)+1$.

Let us consider the locally defined map
$$\Mult: H_{n} \to \C^{n}, \ g(x) \mapsto (g^{\prime}(\beta_1), \dots g^{\prime}(\beta_n))=M(g(x)-x)+(1, \dots , 1),$$
which assigns to $G$ the collection of its multipliers.

\begin{cor}
\label{mult}
 The tangent map $d \Mult: \C^n \to \C^n$  to the multiplier map $\Mult$
  has rank $n-1$ at all points of $H_n$.
\end{cor}

\begin{rem}
The non-triviality of the kernel of $d \Mult$ is related to the fact
that  for
 each (small) complex number
 $\epsilon$
the maps $z \mapsto g(z)$ and $z \mapsto g(z+\epsilon)-\epsilon$ are
conjugate and have the same collection of multipliers.
\end{rem}

In the course of the proof of Theorem \ref{main1} we use the
following purely algebraic assertion that   has a certain
independent interest.

\begin{thm}
\label{algebra} Let $n\ge 2$ be an integer and $f(x)$  a complex
degree $n$ polynomial.  Suppose that there exists a nonzero complex
polynomial $p(x)$ of degree $\le n-2$ such that
$f^{\prime\prime}(x)p(x)-f^{\prime}(x)p^{\prime}(x)$ is divisible by
$f(x)$. Then:
\begin{itemize}
\item[(i)]
 $n \ge 4, \ \deg(p) \ge 2$.
\item[(ii)]
 There exists a quadratic polynomial $\tilde{p}(x)$ such that
$f(x)$ is divisible by $\tilde{p}(x)^2$. In particular, $f(x)$ has a
multiple root.
\item[(iii)]
If $\deg(p)=2$ then
 $p(x)$ divides both $f(x)$ and $f^{\prime}(x)$; in particular, all
 the roots of $p(x)$ are multiple
 roots of
  $f(x)$.
\item[(iv)] If $n=4$ then $\deg(p)=2$  and there exists a nonzero complex number
$c$ such that $f(x)=c\cdot p(x)^2$.
\end{itemize}
\end{thm}

\section{Differentiation}
\label{dif} The first question that naturally arises is how to deal
with $M$? We interpret the ordering of the roots as a choice of an
isomorphism of commutative semisimple
  $\C$-algebras
$$\psi:\Lambda=\C[x]/f(x)\C[x] \cong \C^n, \ u(x)+f(x)\cdot \C[x] \mapsto u(\alpha):=(u(\alpha_1), \dots u(\alpha_n))$$
and carry out all the computations, including the differentiation
with respect to $a$, of functions that take  values in the algebra
$\Lambda$, despite of the fact that this algebra does depend on the
coefficients $a$! (However, its isomorphism  class does not depend
on the coefficients.)
  Of course, while differentiating, we
will use Leibnitz's rule and  that  $f(x)=0$ in $\Lambda$. In what
follows we will often mean under polynomials their images in
 $\Lambda$ ( i.e., the collection of their values at the roots of $f(x)$,
 while we try not refer to the roots explicitly). Notice that the absence of multiple
 roots means that
 $f^{\prime}(x)$ is an invertible element of
 $\Lambda$. Notice also that $\alpha=(\alpha_1, \dots ,
\alpha_n) \in \C^n$ is the image (under $\psi$) of the independent
variable $x$.

The first thing that we want to compute is the derivatives
  $d\alpha/da_i$. Since $f(\alpha)=0$,  $df(\alpha)/da_i=0$. We have
 $$df(\alpha)/da_i= \frac{\partial f}{\partial a_i}(\alpha) + f^{\prime}(\alpha)\cdot d\alpha/da_i.$$
 Since $\partial f/\partial a_i=x^i$, we obtain that
 $$0=\alpha^i+ f^{\prime}(\alpha)\cdot d\alpha/da_i,$$ which gives
 us
 $$d\alpha/da_i=-  [f^{\prime}(\alpha)]^{-1} \alpha^i.$$ It follows
 that for any polynomial
  $u(x)$ (whose coefficients may depend on $a$)
 $$du(\alpha)/da_i= \frac{\partial u}{\partial a_i}(\alpha)  +u^{\prime}(\alpha)\times \{-  [f^{\prime}(\alpha)]^{-1} \alpha^i\}.$$
 We are interested in the case when
 $$u(x)=f^{\prime}(x)=n x^{n-1}+\sum_{i=1}^{n-1}i a_i x^{i-1}.$$ We
 obtain that
 $$d f^{\prime}(\alpha)/d a_i =i \alpha^{i-1} - [f^{\prime}(\alpha)]^{-1} \alpha^i  f^{\prime\prime}(\alpha)$$
 (of course, if $i=0$ then the first term disappears).

Actually, the rank of $dM$ at $f(x)$ is the dimension of the
subspace of   $\Lambda$ generated by
  $n$ elements $\frac{d f^{\prime}}{d a_i}(\alpha)$.
 Suppose that a collection of  $n$ complex numbers $c_0, \dots , c_{n-1}$ satisfies  $\sum_{i=0}^{n-1} c_i \frac{d f^{\prime}}{d a_i}(\alpha)=0$ in $\Lambda$.
 If we put $p(x)=\sum_{i=0}^{n-1} c_i x^i$ then one may easily observe that $p^{\prime}(x)=\sum_{i=1}^{n-1} i c_i x^{i-1}$ and in
 $\Lambda$ the equality
 $$0=\sum_{i=0}^{n-1} c_i \frac{d f^{\prime}}{d a_i}(\alpha)=p^{\prime}(\alpha)- [f^{\prime}(\alpha)]^{-1}p(\alpha)  f^{\prime\prime}(\alpha)$$
 holds. Multiplying (without loss of generality) this equality by
 the invertible element
   $f^{\prime}(\alpha)$, we obtain the equivalent condition:
 $f^{\prime}(\alpha)p^{\prime}(\alpha)-p(\alpha)f^{\prime\prime}(\alpha)=0$
 in
 $\Lambda$. In other words, the polynomial
  $f^{\prime}(x)p^{\prime}(x)-p(x)f^{\prime\prime}(x)$ is divisible
  by
  $f(x)$. Now it is clear that the rank of
  $dM$ at $f(x)$ equals $n-\{$dimension of the space of polynomials  $p(x)$ of degree $\le
  n-1$ such that
  $f^{\prime}(x)p^{\prime}(x)-p(x)f^{\prime\prime}(x)$ is divisible by $f(x)\}$. Obviously, this space contains nonzero  $f^{\prime}(x)$,
which implies that the rank of
   $dM$ always does not exceed $n-1$.
 Since the degree of $f^{\prime}(x)$ is $n-1$, it is easy to observe
 that
 the kernel of $dM$ at $f(x)$ coincides with the direct sum
  $\C\cdot f^{\prime}(x)\oplus W(f)$.   It follows readily that the rank of
  $dM$ at $f(x)$ equals
 $$(n-1)-\dim_{\C}(W).$$

{\bf Proof of Theorem  \ref{main1}} (modulo Theorem \ref{algebra}).
Since $f(x)$ has no multiple roots, it follows from Theorem
\ref{algebra} that
$$W=W(f)=\{0\}$$ and therefore the rank of
 $dM$ at
 $f(x)$ equals
 $$(n-1)-\dim_{\C}(W)=n-1-0=n-1.$$

 \begin{rem}
 \label{general}
A priori, it is clear that the set of all $f(x)$ at which the rank
of $dM$ reaches its maximum value is a non-empty Zariski-open subset
of $P_n$. That is why if we are interested only in the {\sl general
position case} then it suffices to check that the rank is $n-1$ at
least for one $f(x)$: this would imply that the rank is $n-1$ for
{\sl typical} polynomials, i.e., for all polynomials that belong to
a certain non-empty Zariski-open subset of $P_n$.
 \end{rem}

\section{Polynomial Algebra}
{\bf Proof of Theorem \ref{algebra}}. So, let   $f(x)$ be a complex
polynomial of degree $n\ge 2$ and  $p(x)$ a nonzero polynomial of
degree  $\le n-2$ such that
$$R(x):= f^{\prime\prime}(x) p(x) -
f^{\prime}(x)p^{\prime}(x)$$ is divisible by $f(x)$. Without loss of
generality, we may and will assume that both   $f(x)$ and $p(x)$ are
monic, i.e., their leading coefficients  are equal to $1$. Notice
that there is no ``cancelation of degrees" in the expression
$f^{\prime\prime}(x)p(x)-f^{\prime}(x)p^{\prime}(x)$
 (since $p(x)$ and $f^{\prime}(x)$ have distinct degrees).
 Indeed  $\deg(R)\le n+\deg(p)-2$ and the coefficient at
  $x^{n+\deg(p)-2}$ of $R(x)$ equals
 $$n(n-1)-n \deg(p)=n[(n-1)-\deg(p)] \ne 0,$$
 because $\deg(p) \le n-2<n-1$. It follows that
 $R(x)\ne 0$ and
 $$\deg(R)=n-2+\deg(p) \le n-2+n-2=2n-4.$$ (In addition,
 $n[(n-1)-\deg(p)]$ is the leading coefficient of
 $R(x)$.)
 Since $R(x)$ is a nonzero polynomial divisible by  $f(x)$,
 $$n=\deg(f) \le \deg(R)=n-2+\deg(p),$$
 that is why $\deg(p) \ge 2$ and therefore
 $$n-2 \ge \deg(p) \ge 2,$$
 which implies that $n \ge 4$. This proves (i).
 It is also clear that if $n=4$ then $\deg(p)=2$.

 Let us assume now that  $p(x)$ is a monic quadratic polynomial. Then
 $$\deg(R)=n-2+\deg(p)=n-2+2=n=\deg(f)$$and therefore there exists a
 nonzero constant
  $h_0$ such that
$$R(x)=f^{\prime\prime}(x) p(x) -
f^{\prime}(x)p^{\prime}(x)=h_0 f(x).$$ Since $f(x)$ is monic while
the leading coefficient of  $R(x)$ is
$$n[(n-1)-\deg(p)]=n(n-1-2)=n(n-3),$$
we obtain that $h_0=n(n-3)$ and
$$R(x)=f^{\prime\prime}(x) p(x) -
f^{\prime}(x)p^{\prime}(x)=n(n-3)f(x).$$ Differentiating, we obtain
that
$$R^{\prime}(x)=[f^{\prime\prime\prime}(x) p(x)+f^{\prime\prime}(x)p^{\prime}(x)]-
[f^{\prime\prime}(x)p^{\prime}(x)+f^{\prime}(x)p^{\prime\prime}(x)]=n(n-3)f^{\prime}(x).$$
Since $p(x)$  is a monic quadratic polynomial,
$p^{\prime\prime}(x)=2$. Taking this into account,  opening the
parentheses  and grouping together like terms in the formula for
$R^{\prime}(x)$, we obtain that
$$R^{\prime}(x)=f^{\prime\prime\prime}(x)\cdot p(x)-2f^{\prime}(x)=n(n-3)f^{\prime}(x),$$
which gives us
$$f^{\prime\prime\prime}(x)\cdot p(x)=(n^2-3n+2)f^{\prime}(x).$$
We obtain that $p(x)$ divides $f^{\prime}(x)$.  (Recall that $n \ge
4$ and therefore $n^2-3n+2=(n-1)(n-2) \ne 0$.) Since
$$f^{\prime\prime}(x)\cdot p(x) -
f^{\prime}(x)p^{\prime}(x)=n(n-3)f(x),$$ we obtain that $p(x)$
divides $f(x)$. (Recall that $n \ge 4$ and therefore $n(n-3) \ne
0$.) This proves (iii). It is also clear that if (a quadratic
polynomial) $p(x)$ has no multiple roots then $f(x)$ is divisible by
$p(x)^2$. On the other hand, if  $p(x)$ has a double root $\beta$,
i.e., $p(x)=(x-\beta)^2$ then $f(x)$ is divisible by $(x-\beta)^3$.

Let us prove  assertions (ii) and (iv) of Theorem \ref{algebra}. If
all the roots of $f(x)$ are multiple then  (ii) holds true. Let us
assume that $f(x)$ has a simple root and denote it by $\alpha$. Then
$f^{\prime}(\alpha) \ne 0$  and there exists a complex polynomial
$f_1(x)$ of degree $n-1$ such that
$$f(x)=(x-\alpha) f_1(x).$$ The simplicity of $\alpha$ means that
$$f_1(\alpha) \ne 0, \  f^{\prime}(\alpha) \ne 0.$$
Let  $V$ be the two-dimensional vector (sub)space of polynomials
generated by  $f^{\prime}(x)$ and $p(x)$. (It is two-dimensional,
because nonzero $f^{\prime}(x)$ and $p(x)$ have different degrees.)
Clearly, the degree of any polynomial from  $V$ does not exceed
$n-1$. It is also clear that for each $q(x)\in V$ the polynomial
$f^{\prime\prime}(x) q(x) - f^{\prime}(x)q^{\prime}(x) $ is
divisible by $f(x)$, since this is true for both $q(x)=p(x)$ and
$q(x)=f^{\prime}(x)$. Choose in the two-dimensional $V$ {\sl a
nonzero} polynomial $q(x)$ such that
$$q(\alpha)=0.$$ Clearly, $$q(x) \not\in \C\cdot f^{\prime}(x), \ \deg(q) \ \le
n-1.$$ We have
$$f^{\prime\prime}(x) q(x) - f^{\prime}(x)q^{\prime}(x)=h(x)f(x) \eqno{(*)}$$
for some complex polynomial $h(x)$. Since
$$f(\alpha)=0, \ f^{\prime}(\alpha)\ne 0, \ q(\alpha)=0,$$
we conclude that $q^{\prime}(\alpha)=0$, i.e., $\alpha$ is a
multiple root of nonzero  $q(x)$, hence, there exists a nonzero
complex polynomial $q_1(x)$ such that
$$q(x)=(x-\alpha)^2 q_1(x), \ \deg(q_1)=\deg(q)-2\le (n-1)-2=\deg(f_1)-2.$$
I am going to prove that ${f_1}^{\prime\prime}(x) q_1(x) -
{f_1}^{\prime}(x){q_1}^{\prime}(x)$ is divisible by $f_1(x)$ (and
apply induction by $n$).

 We have
$$f(x)=(x-\alpha)f_1(x), \
f^{\prime}(x)=f_1(x)+(x-\alpha){f_1}^{\prime}(x), \
f^{\prime\prime}(x) = 2 {f_1}^{\prime}(x)+
(x-\alpha){f_1}^{\prime\prime}(x),$$
$$q(x)=(x-\alpha)^2 q_1(x), \ q^{\prime}(x)=(x-\alpha)^2
{q_1}^{\prime}(x)+2(x-\alpha)  q_1(x).$$ Plugging in these
expressions in  (*), we obtain
$$[2 {f_1}^{\prime}(x)+
(x-\alpha){f_1}^{\prime\prime}(x)](x-\alpha)^2
q_1(x)-[f_1(x)+(x-\alpha){f_1}^{\prime}(x)][(x-\alpha)^2
{q_1}^{\prime}(x)+2(x-\alpha)  q_1(x)]=$$ $$=h(x)(x-\alpha)f_1(x).$$
Dividing both sides by $(x-\alpha)$, we obtain
$$[2 {f_1}^{\prime}(x)+
(x-\alpha){f_1}^{\prime\prime}(x)]
q_1(x)(x-\alpha)-[f_1(x)+(x-\alpha){f_1}^{\prime}(x)][(x-\alpha)
{q_1}^{\prime}(x)+2 q_1(x)]=h(x)f_1(x).$$ Moving to the right hand
side all the terms containing $f_1(x)$, we obtain that
$$[2 {f_1}^{\prime}(x)+
(x-\alpha){f_1}^{\prime\prime}(x)] q_1(x)(x-\alpha)-
(x-\alpha){f_1}^{\prime}(x)[(x-\alpha) {q_1}^{\prime}(x)+2
q_1(x)]=$$ $$=\{h(x)+[(x-\alpha) {q_1}^{\prime}(x)+2
q_1(x)]\}f_1(x).$$ If we put $h_1(x):=h(x)+[(x-\alpha)
{q_1}^{\prime}(x)+2 q_1(x)]$ then we get
$$[2 {f_1}^{\prime}(x)+
(x-\alpha){f_1}^{\prime\prime}(x)] q_1(x)(x-\alpha)-
(x-\alpha){f_1}^{\prime}(x)[(x-\alpha) {q_1}^{\prime}(x)+2
q_1(x)]=h_1(x)f_1(x).$$

The left hand side is divisible by $(x-\alpha)$,  while  $f_1(x)$
(in the right hand side) is not divisible by $(x-\alpha)$. This
means that there exists a complex polynomial $h_2(x)$ such that
$h_1(x)=(x-\alpha)h_2(x)$ and
$$[2 {f_1}^{\prime}(x)+
(x-\alpha){f_1}^{\prime\prime}(x)]
q_1(x)-{f_1}^{\prime}(x)[(x-\alpha) {q_1}^{\prime}(x)+2
q_1(x)]=h_2(x)f_1(x).$$ Opening the parentheses  in the left hand
side and grouping together like terms, we obtain
$$(x-\alpha){f_1}^{\prime\prime}(x)
q_1(x)-{f_1}^{\prime}(x)(x-\alpha) {q_1}^{\prime}(x)=h_2(x)f_1(x).$$
Again  the left hand side is divisible by $(x-\alpha)$, while
$f_1(x)$ (in the right hand side) is not divisible by $(x-\alpha)$.
This means that there exists a complex polynomial $h_3(x)$ such that
 $h_2(x)=(x-\alpha)h_3(x)$ and
$${f_1}^{\prime\prime}(x) q_1(x)-{f_1}^{\prime}(x)
{q_1}^{\prime}(x)=h_3(x)f_1(x).$$ Let us treat separately the cases
 $n=4$ and $n>4$.

Suppose that  $n=4$. Then $\deg(f_1)=4-1=3$ and we get a
contradiction to the already proven  (i). Therefore all the roots of
 $f(x)$ are multiple, which implies that
 $f(x)$ has  either two double roots or one root of multiplicity $4$.  (In particular $f(x)$
 has no roots of multiplicity $3$.)
We also know that $\deg(p)=2$ while $p(x)$ divides both $f(x)$ and
$f^{\prime}(x)$. We have seen that if quadratic $p(x)$ has no
multiple   roots then $f(x)$ is divisible by $p(x)^2$ and the
comparison of degrees and leading coefficients implies that $f(x)=
p(x)^2$. On the other hand, if $p(x)$ has a double root $\beta$ then
we have seen that $p(x)=(x-\beta)^2$ and $f(x)$ is divisible by
$(x-\beta)^3$. Since $f(x)$ has no roots of multiplicity $3$, we
conclude that monic $f(x)=(x-\beta)^4$. Since  $p(x)=(x-\beta)^2$ we
conclude that $f(x)= p(x)^2$. This proves (iv). In addition, it also
proves (ii) for $n=4$.


Suppose that $n>4$. Then $\deg(f_1)=n-1$. Using induction by $n$, we
may apply the assertion (ii) of Theorem \ref{algebra} to $n-1$
(instead of $n$) and $f_1$ and $q_1$ (instead of $f$ and $p$
respectively) and obtain that there exists a quadratic polynomial
$\tilde{p}(x)$ such that $f_1(x)$ is divisible by $\tilde{p}(x)^2$.
It follows that $f(x)=(x-\alpha)f_1(x)$ is also divisible by
$\tilde{p}(x)^2$. This ends the proof.

\begin{rem} Inspecting the proof of Theorem
 \ref{algebra}, we observe that it works (not only over arbitrary fields of characteristic zero but also) over fields of characteristic $>n$.
In particular, the assertion of Theorem
   \ref{algebra} remains true over fields of positive characteristic $>n$.
\end{rem}

\begin{ex}
Let $p(x)$ be a monic quadratic polynomial. Let us put $n=4$ and
$f(x)=p(x)^2$. Then
$$p^{\prime\prime}(x)=2, \ f^{\prime}(x)=2 p(x) p^{\prime}(x),$$
$$f^{\prime\prime}(x)=2 p^{\prime}(x)^2+2 p(x) p^{\prime\prime}(x)=2
p^{\prime}(x)^2+4 p(x).$$ Then
$$f^{\prime\prime}(x) p(x) -
f^{\prime}(x)p^{\prime}(x)=[2 p^{\prime}(x)^2+4 p(x)]p(x)-2 p(x)
p^{\prime}(x)p^{\prime}(x)=4 p(x)^2 = 4 f(x).$$
\end{ex}

\begin{ex}
Let $d$ be an arbitrary complex number. Let us put $n=5$,
$$f(x):=(x+d)\left(x^2 -\frac{d}{2}x+\frac{9}{4} d^2\right)^2=
x^5+\frac{15}{4}d^2 x^3+\frac{5}{2}d^3 x^2+\frac{45}{16}d^4
x+\frac{81}{16}d^5,$$
$$p(x):= x^3+2d^2 x+\frac{9}{8}d^3.$$ Then
$$f^{\prime\prime}(x) p(x) - f^{\prime}(x)p^{\prime}(x)=5x\cdot f(x).$$
\end{ex}

\begin{rem}
\label{genericcase} If   $f(x)$ is the polynomial $x^{n}-1$, which
has no multiple roots, then one may directly check that there does
{\sl not} exist a nonzero polynomial  $p(x)$ of degree $\le n-2$
such that
$$R(x)= f^{\prime\prime}(x) p(x) - f^{\prime}(x)p^{\prime}(x)$$ is
divisible by  $f(x)$, i.e., $R(x)=f(x)h(x)$ for some polynomial
 $h(x)$. Indeed, let us assume that such a  $p(x)$ does exist; we
 may assume that  $p(x)$ is monic. Obviously
 (see the very beginning of the proof of theorem
 \ref{algebra})
$$R(x)\ne 0, \ \deg(R) \le 2n-4.$$ This implies that
 $h(x)\ne 0$ and
$$\deg(h)=\deg(R)-n \le n-4.$$
Plugging in into the formula for
  $R(x)$ explicit formulas for $f(x)$ and its first and second derivatives, we obtain that
 $$ n(n-1) x^{n-2} p(x)-n x^{n-1} p^{\prime}(x)=h(x) (x^n-1).$$
 This implies that nonzero
$h(x)$ is divisible by $x^{n-2}$ and therefore $\deg(h)\ge n-2$,
which is not true. The obtained contradiction proves that such a
$p(x)$ does not exist and therefore
 $W(f)=\{0\}$, i.e., the rank of
 $dM$ at $f(x)=x^n-1$ has rank $n-1$. According to Remark
 \ref{general}, it follows immediately that the rank of
 $dM$ equals $n-1$ for
{\sl typical} polynomials.
\end{rem}

\begin{rem}

After having read previous versions of this paper, Elmer Rees
\cite{Rees} and (independently and simultaneously) Victor S. Kulikov
outlined a more direct way of approaching the proof of Theorem
\ref{main1}. Namely, they consider the Jacobi matrix of the map
$$\widetilde{M}: (\alpha_1, \dots , \alpha_n) \mapsto (f^{\prime}(\alpha_1), \dots ,
f^{\prime}(\alpha_n))$$ and prove that its first principal
 $(n-1)\times (n-1)$ minor equals
$$c \prod_{1 \le i<j<n}(\alpha_i-\alpha_j)^2,$$ where $c$ is a
constant that does not depend on  $f(x)$. (In order to prove Theorem
\ref{main1} it suffices to check that $c \ne 0$.) Rees \cite{Rees}
proved
that  $$c =(-1)^{(n-1)(n-2)/2} (n-1)!.$$ Kulikov's approach is
presented in the Appendix to this paper that he has kindly agreed to
contribute.

\end{rem}

\vskip 1cm

 {\bf Acknowledgements.} This note has arisen from an attempt to answer a
 question posed by Yu.S. Ilyashenko in connection with  \cite{Ilyash,Igor}. (In fact, he was interested in
 the rank of $dM$ for {\sl typical} polynomials $f(x)$.) I thank him for his questions,
stimulating discussions and interest in this paper. I also thank
E.E. Shnol, whose papers, lectures and seminars developed my
interest in algebraic aspects of ODE. I am grateful to Vladimir L.
Popov and Victor S. Kulikov, who read a preliminary version of this
note and made several useful remarks.

\section[Appendix by Victor S. Kulikov]{Appendix by Victor S. Kulikov \protect\footnote{Steklov Mathematical Institute of the Russian Academy of Sciences,
 email: kulikov\char`\@mi.ras.ru} }
Recall (see Remarks \ref{general} and \ref{genericcase}) that it is
much more easier to prove  that the differential $dM$ has rank $n-1$
at a generic point rather than  Theorem \ref{main1} in full. This
observation allows us to deduce Theorem \ref{main1} from the
following statement that is of independent interest and whose proof
requires minimal computations.
 Let us consider a polynomial
$$f(x,\alpha)=(x-\alpha_1)\dots (x-\alpha_n)=x^n+\sum_{i=0}^{n-1}a_ix_i,$$ where  $\alpha
=(\alpha_1\dots,\alpha_n)$. Let us define the map $\widetilde
M:\C^n\to \C^n$ by
 $$\widetilde
M(\alpha)=(f_x'(\alpha_1,\alpha),\dots, f_x'(\alpha_n,\alpha)).$$
Let $D$ be the Jacobi matrix of the map
 $\widetilde M$.
\begin{thm}
\label{minor} The rank $\rg\, D$ of the matrix $D$ is $\leq n-1$.
The vector $v=(1,\dots,1)$ is an eigenvector of  $D$ with eigenvalue
$0$. If
 $n>1$ then for each $i$, $1\leq i\leq n$ the minor $\M_{i,i}$
of $D$ equals $c \prod_{j<k,j\neq i,k\neq i}(\alpha_k-\alpha_j)^2$,
where $c$ is a certain nonzero constant
\end{thm}

Clearly, Theorem \ref{main1} is a corollary of Theorem \ref{minor}.
Indeed, first  $\M_{i,i}\neq 0$ at every point with pairwise
distinct coordinates. Second, for each $\alpha$ the line
$l_{\alpha}(t)=\alpha-t(1,\dots,1)$ lies in a fiber of the map
 $\widetilde M$. Therefore the curve
$f(x+t)$ lies in a fiber of the map $M$ and the tangent vector to
this curve at the point $f(x)$ is $f'(x)$.

\begin{proof}[Proof of Theorem \ref{minor}] The map $\widetilde M$ with respect to the coordinates $(\alpha_1,\dots,
\alpha_n)$ is defined by polynomials
 $y_i=f_x'(\alpha_i,\alpha)=(\alpha_i-\alpha_1)\dots
(\alpha_i-\alpha_{i-1})(\alpha_i-\alpha_{i+1})\dots
(\alpha_i-\alpha_n)$, $1\leq i\leq n$. Computing the partial
derivatives, one may easily observe that the entry
 $d_{i,j}$ of $D$ equals
$d_{i,j}=-\frac{f_x'(\alpha_i,\alpha)}{(\alpha_i-\alpha_j)}$ if
$j\neq i$ and $d_{i,i}=\sum_{j\neq
i}\frac{f_x'(\alpha_i,\alpha)}{(\alpha_i-\alpha_j)}$. Therefore the
sum of all columns of $D$ equals zero. That is why  $\det D=0$ and
the vector вектор $v=(1,\dots,1)$ is an eigenvector of $D$ with
eigenvalue $0$.

Let us prove that the minor  $\M_{i,i}$ of $D$ equals $c
\prod_{j<k,j\neq i,k\neq i}(\alpha_k-\alpha_j)^2$. When $n=2$ this
is obvious.

Suppose that  $n>2$. First notice that $\M_{i,i}$ is a polynomial in
variables  $\alpha_j$ of degree $(n-1)(n-2)$. Besides, one may
easily observe that the minors  $\M_{i,i}$ are invariant with
respect to  transformations
 $\sigma_{j,k}$ defined as follows:
 an exchange of $j$th and $k$th rows followed by the
 exchange of columns with the same numbers and a change of
 variables  $\alpha_j\leftrightarrow \alpha_k$. Since the operation
 of exchanging
  two columns and two rows of a matrix does not change
 its determinant,  $\M_{i,i}$ is a symmetric function in variables
 $\alpha_j$, $j\neq i$. Further, let us restrict
 $\widetilde M$ to the hyperplane
$\alpha_i=\const$ and consider the composition of this restriction
with  the projection map to the hyperplane $y_i=0$: $(y_1,\dots
,y_n)\mapsto (y_1,\dots ,y_{i-1},y_{i+1},\dots, y_n)$. Denote the
obtained map by $\widetilde M_i:\C^{n-1}\to\C^{n-1}$. Clearly, the
Jacobi matrix of $\widetilde M_i$ coincides with the matrix $D_i$
obtained from $D$ by deleting its $i$th row and $i$th column. It is
easy to observe that for each triple  $i,j,k$ of pairwise distinct
indices the image of the hyperplane $\alpha_j=\alpha_k$ under
$\widetilde M_i$ has codimension $\geq 2$, because it lies in the
intersection of the hyperplanes  $y_j=0$ and $y_k=0$.  Therefore the
jacobian of $\widetilde M_{i}$ vanishes at
 $\alpha_j=\alpha_k$ and any value of $\alpha_i$. On the other hand,
 the jacobian of
 $\widetilde M_{i}$ coincides with
$\M_{i,i}$ and therefore $\M_{i,i}$ is divisible by
$\alpha_j-\alpha_k$. Since $\M_{i,i}$ is a symmetric polynomial in
variables  $\alpha_j$, $j\neq i$, we obtain that $\M_{i,i}$ is
divisible by  $(\alpha_j-\alpha_k)^2$ for each triple $i,j,k$ of
pairwise distinct indices.  Taking into account that
$\deg(\M_{i,i})=(n-1)(n-2)$, we obtain that $\M_{i,i}=c
\prod_{j<k,j\neq i,k\neq i}(\alpha_k-\alpha_j)^2$ for a certain
constant константы $c$.  Since the vector $(1,\dots,1)$ is not
tangent to the hyperplane $\alpha_i=\const_i$, one may easily
observe that  $c\neq 0$ because $dM$ (and therefore $d\widetilde M$)
has rank  $n-1$ at a generic point.
\end{proof}

\begin{thm} The image of the map $\widetilde M:\C^n\to \C^n$
lies in the irreducible hypersurface
  $S_{n-1}$ that is defined in $\C^n$ by the equation
\begin{equation} \label{eq1} \sum_{i=1}^n y_1\dots
y_{i-1}y_{i+1}\dots y_n=0.\end{equation} The image $\widetilde
M(\mathbb C^n)$ is everywhere dense in  $S_{n-1}$ with respect to
the complex topology.

\end{thm}

\begin{proof}
The irreducibility of $S_{n-1}$  follows from the irreducibility of
the polynomial
$$s_{n-1}=\sum_{i=1}^n  y_1\dots y_{i-1}y_{i+1}\dots y_n,$$
which can be easily checked.
The map $\widetilde M$ is defined by polynomials
$$y_1=f_x^{\prime}(\alpha_1,\alpha),\dots,
y_n=f_x^{\prime}(\alpha_n,\alpha).$$ Let us view $s_{n-1}$ as a
polynomial in   $\alpha_1,\dots,\alpha_n$, plugging in
$y_i=f_x^{\prime}(\alpha_i,\alpha)$.
We observe that $s_{n-1}$ is a homogeneous symmetric polynomial in
$\alpha_1,\dots,\alpha_n$ of degree $(n-1)^2$.

Clearly, if $\alpha_i=\alpha_j$ then $y_i=y_j=0$. Therefore the
polynomial $s_{n-1}$ vanishes if $\alpha_i=\alpha_j$  and therefore
it is divisible by  $\alpha_i-\alpha_j$ for each $i$ and $j$, $i\neq
j$. Since $s_{n-1}$ is symmetric, it is divisible by
$(\alpha_i-\alpha_j)^2$. Therefore $s_{n-1}$ is divisible by the
polynomial
 $\prod_{i\neq
j}(\alpha_i-\alpha_j)$, whose degree is $n(n-1)$. Since
$n(n-1)>(n-1)^2$, we have $s_{n-1}(\alpha)\equiv 0$ and therefore
$\widetilde M(\mathbb C^n)$ lies in $S_{n-1}$.
Let $X$ be the open algebraic subvariety of all points in   $\C^n$
with pairwise distinct coordinates and let $Y \subset S_{n-1}$ be
the Zariski closure of the (sub)set $\widetilde M(X)$. Clearly, $Y$
is an irreducible closed algebraic subvariety in  $S_{n-1}$; in
particular,  $\dim(Y)\le \dim(S_{n-1})=n-1$. It follows from
\cite[Ch. 1, Sect. 5, Th. 6]{Sh}  that
 $\widetilde M(X)$ contains a Zariski-open non-empty subset of
 $Y$; in particular, it contains a nonsingular point of $Y$.
Theorem \ref{minor} combined with irreducibility of  $S_{n-1}$
implies that $\dim(Y)\ge n-1$. It follows that $\dim(Y)=
n-1=\dim(S_{n-1})$ and therefore $Y=S_{n-1}$. Since the image
  $\widetilde M(\C^n)$ contains $\widetilde M(X)$, we conclude that
 $\widetilde M(\C^n)$ contains a Zariski-open non-empty subset of
 $S_{n-1}$; in particular,  $\widetilde M(\mathbb C^n)$ is
 everywhere dense in
$S_{n-1}$ with respect to the complex topology.

\end{proof}

 \end{document}